\documentclass[12pt]{article}
\usepackage{amsmath,amscd,amsfonts,amssymb,latexsym,amsthm}
\usepackage[utf8]{inputenc}
\usepackage{graphicx}
\usepackage[OT4]{fontenc}
\usepackage{bbm}
\usepackage{stmaryrd}
\usepackage[normalem]{ulem}

\DeclareMathAlphabet{\mathmybb}{U}{bbold}{m}{n}

\newtheorem{theorem}{Theorem}

\newtheorem{lemma}[theorem]{Lemma}
\newtheorem{corollary}[theorem]{Corollary}

\newcommand{\beq}{\begin{equation}}
	\newcommand{\eeq}{\end{equation}}

\newcommand{\E}{\mathbb{E}}

\renewcommand{\d}{\partial}

\DeclareFontFamily{U}{mathx}{\hyphenchar\font45}
\DeclareFontShape{U}{mathx}{m}{n}{
	<5> <6> <7> <8> <9> <10>
	<10.95> <12> <14.4> <17.28> <20.74> <24.88>
	mathx10
}{}
\DeclareSymbolFont{mathx}{U}{mathx}{m}{n}
\DeclareMathSymbol{\bigtimes}{1}{mathx}{"91}

\renewcommand{\vec}[1]{\begin{bmatrix}#1\end{bmatrix}}

\newcommand{\ilsk}[2]{\left\langle #1\mid #2\right\rangle}

\newcommand{\he}[2]{\mathrm{He}_{#1}\left(#2\right)}

\def\R{\mathbb{R}}

\def\N{\mathbb{N}}

\def\sup{{\rm sup}}

\title{B-splines approximate the Gaussian density and Hermite functions in Schwartz seminorms}
\author{Maciej Rzeszut$^1$, Michał Wojciechowski$^2$}
\date{%
    $^1$Jan Kochanowski University\\%
    $^2$Institute of Mathematics, Polish Academy of Sciences\\[2ex]%
    \today}
\begin{document}
\maketitle 
\abstract{We prove that B-splines with knots satisfying assumptions of the Berry-Esseen Theorem \cite[p.542]{zbMATH03349081}, which correspond  directly to the volumes of sections of the standard simplex, tend to the Gaussian density in any Schwartz seminorm. As a consequence, we derive some asymptotic formulas connecting Hermite and Laguerre polynomials. }
\par
%Legal disclaimer: multiple constants and orders of magnitude were harmed in the making of this note. The act of denoting local variables meaning different things in separate lemmas by the same letter is meant solely to confuse \sout{the enemy} the reader.

{\let\thefootnote\relax\footnotetext{This research was partially supported by the National Science Center, Poland, and Austrian Science Foundation FWF joint CEUS programme. National Science Center project no. 2020/02/Y/ST1/00072.}}

\section{Introduction}
For a given sequence of real numbers $x_1<\ldots<x_N$ (which may be taken equal by taking a limit), one may define the spline space of order $n-1$ as the space of functions that are polynomials of degree $\leq n-2$ on each interval $\left[x_j,x_{j+1}\right]$ and $C^{n-3}$ at each of the points $x_1,\ldots,x_m$ (see e.g \cite{B} for more details). This space has a basis consisting of the so-called B-splines $B_{x_j,\ldots,x_{j+n-1}}$ defined by the condition of being $0$ outside of $\left[x_j,x_{j+n-1}\right]$ for $j=1,\ldots,N-n+1$. These B-splines are unique up to multiplication by a constant and do not depend on other $x_i$'s, so one can take $N=n$ and think just about $B_{x_1,\ldots,x_n}$. There happens to be an explicit formula:
\[
B(t)=\sum_{k=1}^n \frac{(x_k-t)_+^{n-2}}{\prod_{j\neq k}(x_k-x_j)}.
\]
For equally spaced $x_i$'s, $B(t)$ turns out to be, up to an affine substitution, an $(n-1)$-fold convolution of the uniform density on $[0,1]$. Thus, according to the local limit theorem (cf. \cite{Kh}, \cite{GK}) it tends to the Gaussian density. No result of this type was previously known for knots that are not equally spaced. Note that local limit theorem is a step further than the classical CLT, because instead of considering convergence of cumulative distribution functions, we deal with convergence of densities, i.e. their derivatives. The aim of the present paper is to prove that the analogous fact  holds for a generic choice of $x_i$'s and it does so in any Schwartz seminorm. That is, we have

\begin{theorem} \label{mainthm}Let $p,q\in \N_0$, $\sum x_k=0$, $\sum x_k^2=1$. Then if $n$ is big enough and $\sum \left|x_k\right|^3$ is small enough (with respect to $p,q$),
%and\footnote{This is not a very restrictive assumption since the sum is $0$ unless $\frac{|t|}{n}\leq \max |x_k|\leq 1$.} $|t|\leq n$, 
then
\[
\left|t^p \d_t^q\left(\frac{1}{\sqrt{2\pi}}e^{-t^2/2} - \sum_{k=1}^n \frac{\left(x_k-\frac{t}{n}\right)_{+}^{n-2} }{\prod_{j\neq k} \left(x_k-x_j\right)}
  \right)  \right| \lesssim_{p,q} \sum_{k=1}^n \left|x_k\right|^3.\]

\end{theorem}

The symbol $\lesssim_{p,q}$, as in $A\lesssim_{p,q}B$, stands for $A\leq C B$, where $C$ is a constant dependent on parameters $p,q$; the symbol $A\eqsim B$ stands for $A\lesssim B$ and $B\lesssim A$.\par
Note here that the quantity $O\left(\sum x_i^3\right)$ that controls the error of an approximation is as small as $n^{-1/2}$ for not too concentrated $x_i$'s. An immediate consequence of the above result follows by the fact that derivatives of Gaussian density are nothing but Hermite functions. So denoting by $\he{r}{t}$ the $r$-th Hermite polynomial we get

\begin{corollary}
    Let $p,q,r\in \N_0$, $\sum x_k=0$, $\sum x_k^2=1$. Then if $n$ is big enough and $\sum \left|x_k\right|^3$ is small enough,

\[
\left|t^p\partial_t^q\left(\frac{1}{\sqrt{2\pi}}\he{r}{t}e^{-t^2/2} - \sum_{k=1}^n \frac{\left(x_k-\frac{t}{n}\right)_{+}^{n-2-r} }{\prod_{j\neq k} \left(x_k-x_j\right)}
  \right) \right| \lesssim_{p,q,r} \sum_{k=1}^n \left|x_k\right|^3.\]

\end{corollary}

%$B(t)$ for $x_i$'s rescaled to $\sum x_i=0$, $\sum x_i^2=1$ approximates the Gaussian density in any Schwartz seminorm (i.e. $\|f\|_{\mathcal{S}^{p,q,\infty}}:=\sup_{t\in \R} \left|t^p\left(\frac{d}{dt}\right)^q f(t)\right|$) with an error of $O\left(\sum x_i^3\right)$. \par

There are several ingredients in the proof of Theorem \ref{mainthm}. The bridge between B-splines and probability theory is done by means of Hermite-Genocchi Formula equating our B-spline to a certain integral over a simplex, and a well known representation (Lemma \ref{unifsimplex}) of the $(n-1)$-dimensional Hausdorff measure on $\Delta_{n-1}\subset \R^n$ by using an exponential distribution. This allows us to write $\left(1-\frac{1}{n}\right)B_{x_1,\ldots,x_n}(t/n)$ as the density of $\frac{Q_1}{1+n^{-1/2}Q_2}$, where $(Q_1,Q_2)= \sum (x_k,n^{-1/2})(P_k-1)$ and $P_1,\ldots,P_n$ are iid Exp(1). It is technical to show that for $N_1,N_2$ iid N(0,1), the density of $N_1$ is close to the density of $\frac{N_1}{1+n^{-\frac12}N_2}$ (Lemma \ref{puregauss}) and to show that the latter and $\frac{Q_1}{1+n^{-1/2}Q_2}$ are close, provided that the joint densities of $(Q_1,Q_2)$ and $(N_1,N_2)$ are close. This is done is Section 4.  
Therefore, the heart of the matter of our proof is to provide the control of the convergence rate of the density of $(Q_1,Q_2)$ to
the Gaussian density of $(N_1,N_2)$ in Schwartz seminorms $\mathcal{S}^{p,q,\infty}$, where $\|f\|_{\mathcal{S}^{p,q,r}}:=\left\|t \mapsto |t|^p \d^q f(t)\right\|_{L^r}$. We do this in Section 3.
Our argument follows a standard approach to local limit theorems, with some deviations. We begin by taking the Fourier transforms of the densities, which need to be close in $\mathcal{S}^{p,q,1}$ rather than in $\mathcal{S}^{p,q,\infty}$ (Theorem \ref{mainl1poly}). To prove this, we need essentially two types of estimates that are put together by means of the Fa\'{a} di Bruno's Formula: closeness in $\mathcal{S}^{p,0,1}$ (Corollary \ref{ezminushpolyl1}) and at most polynomial growth of $\left(\frac{\d}{\d\xi_b}\right)^r\left(\log \varphi_{(Q_1,Q_2)}-\log \varphi_{(N_1,N_2)}\right)(\xi)$ (Lemma \ref{glemma}, \ref{fhlemma}). The latter is technical, the former utilizes tail estimates reminiscent of Khinchine's proof of his LLT (Lemma \ref{khintail}). 

There are some remarks to be made. Firstly, we do not know whether
the estimation of the error is asymptotically sharp. One can check that for equally spaced $x_i$'s optimal error is $O(n^{-1})$ but
our proof does not give anything better than $O(n^{-1/2})$ in any case.\par
Secondly, our proof of Lemma \ref{khintail} seems to require $m^3:=\sum |x_k|^3$ to be small enough depending on the parameters, which at the end of the day are controlled by $p,q$. We do not know if our theorem is true for $m$ close to 1. It should also be noted that one may choose to control the spread of $x_k$'s by means of a weaker norm than $\ell^3$, since if $p>3$, then $\frac{1}{3}=\frac{u}{p}+\frac{1-u}{2}$ for some $u\in (0,1)$ and thus $\frac{\|x\|_{\ell^3}}{\|x\|_{\ell^2}}\leq \left(\frac{\|x\|_{\ell^p}}{\|x\|_{\ell^2}}\right)^u$, so $m\leq \|x\|_{\ell^p}^u$.  \par
Thirdly, the fact that $B(t)$ approaches the Gaussian density should not be a surprise. As we already mentioned, $B(t)$ was identified as the $(n-2)$-dimensional volume of section of the section of the $(n-1)$-simplex by the hyperplane perpendicular to the vector $x=(x_1,x_2,\dots,x_n)$, shifted along this vector. The celebrated Klartag's CLT for convex sets (cf. \cite{K}) states that for any convex body $K$ and for the set of directions $y$ of measure $1-\omega_n$ the function $t\mapsto vol_{n-1}(K\cap \{z:\langle z,y\rangle<t\}$ approximates the cumulative distribution function of Gaussian variable with an error $\varepsilon_n$ where
$(\omega_n),(\varepsilon_n)$ are some sequences tending to zero 
(later in \cite{EK} the estimate was given for probability density function approximation).
Hence the appearance of the Gaussian density in Theorem \ref{mainthm} follows from Klartag's result. Our result could be seen as a quantitative description of approximation in much stronger norm for the specific case of simplex. Besides, to approximate Hermite functions it was necessary to use this stronger estimate of Schwartz seminorms.

Here one should mention that this particular case of simplex was considered even earlier by Brehm and Voigt (cf. \cite{BV}) who obtained the formula for very special directions $x$ (pointing from the center of some subsimplex to the center of the complementary subsimplex) and proved pointwise convergence for most of such directions. One of the advantages of our approach is a precise bound on error in terms of direction.

In the last Section we present some applications of Theorem 1 making use
of the particularly nice behavior of derivation and Fourier transform with respect to the Schwartz seminorms. That is, we present a proof of the following result connecting Hermite polynomials with Laguerre polynomials $L_k^{(\alpha)}$, which apparently does not have a clear connection with B-splines nor with probability. For an extensive discussion on Laguerre polynomials with $\alpha<-1$ see \cite{DGV} and \cite{Sz}.

\begin{corollary}
    Let $p,q,r\in \N_0$, $\sum x_k=0$, $\sum x_k^2=1$. Then if $n$ is big enough and $\sum \left|x_k\right|^3$ is small enough,
    
    \[\left\|\frac{C_{r,n}}{\xi^{n+r-1}}\sum_{k=1}^n \frac{e^{-in\xi x_k}}{\prod_{j\neq k} \left(x_k-x_j\right)   
    }L_r^{(-n-r+1)}(-n\xi x_k) - \he{r}{\xi}\cdot e^{-\xi^2/2}\right\|_{\mathcal{S}^{p,q,\infty}}=O_{p,q}(m^3),\]
where $C_{r,n}=\frac{(n-2)!\cdot r!\cdot i^{r+n-1}}{n^{n-2}}$.
\end{corollary}   

Note that, since Hermite polynomials are real valued, the above corollary gives separate information about both real and imaginary parts of the sum involved. In addition, even the  continuity at 0 of the function at hand is a nontrivial part of the statement.

\subsection{Acknowledgments} The authors thank Markus Passenbrunner
for discussions on B-splines and  Bartłomiej Zawalski for help at the experimental stage of this project.

\section{Notations}
Throughout the paper we use the following notation and assumptions. $x_1,\ldots,x_n$ are real numbers such that $\sum x_k=0$, $\sum x_k^2=1$. Vectors $v_k$ are defined by $v_k= \vec{x_k\\ n^{-\frac{1}{2}}}$. For $b\in\{1,2\}$, we denote $b$-th coordinate of $v_k$ by $v_{kb}$. We use $\|\cdot\|$ for a norm of an $\R^n$ vector and $|\cdot|$ for the length of a $2$-dimensional one. The matrix $V^T=\vec{v_1 & \cdots & v_n}$ satisfies $V^TV=I_2$, which implies $\sum\ilsk{\xi}{v_k}^2=|\xi|^2$ for any $\xi \in \R^2$. We additionally assume $m^3:=\sum\left|v_k\right|^3\eqsim \|x\|_{\ell^3}^3+\frac{1}{\sqrt{n}}$ is small. Moreover, we denote $t_k=\ilsk{\xi}{v_k}$ whenever $\xi$ is present, and 
\[F=F(\xi)=-\frac{1}{2}\sum\ln\left(1+t_k^2\right),\quad G=G(\xi)=-\sum\left(t_k-\arctan t_k\right),\]
\[H=H(\xi)-\frac{1}{2}|\xi|^2= -\frac{1}{2}\sum t_k^2,\quad Z=Z(\xi)=F+iG\]
and we will drop the indication of the argument, which will always be $\xi$, in order to keep the notation more concise. 
 For $b\in\{1,2\}$, the symbol $\d_b$ denotes $\frac{d}{d\xi_b}$ as in 
 \[\d_1 t_k= x_k,\quad \d_2 t_k =\frac{1}{\sqrt{n}}.\]
  Our Hermite polynomials will be the probabilists', defined by 
  \[\he{n}{z}e^{-z^2/2}=(-1)^n \left(\frac{d}{dz}\right)^n e^{-z^2/2}.\] The symbol $P$ is reserved for an unspecified polynomial that may change from line to line, much like $C$ denotes a universal constant. If there are parameters in the subscript, then the bound for the polynomial or constant depends on them, e.g. $P_r(|\xi|)$ denotes a polynomial $\leq C_r (1+|\xi|)^r$. \\
  The standard normal distribution on $\R^2$ of density $t\mapsto\frac{1}{2\pi}\exp\left(-|t|^2/2\right)$ will be denoted by $\mathrm{N}(0,I_2)$ and the exponential distribution on $\R$ with parameter 1, i.e. of density $t\mapsto\mathbbm{1}_{[0,\infty)}(t)\exp(-t)$, will be denoted by $\mathrm{Exp}(1)$. \\
  We will denote the standard $n-1$-dimensional simplex in $\R^n$ by $\Delta_{n-1}$, i.e.
  \[\Delta_{n-1}=\left\{x\in \R_+^n: \sum_{j=1}^n x_j=1\right\}. \]
  As this is an $n-1$-dimensional submanifold of $\R^{n-1}$, we may equip it with the $n-1$-dimensional Hausdorff measure $\mathcal{H}^{n-1}$ and denote the normalized measure 
  \[A\mapsto  \frac{\mathcal{H}^{n-1} \left(A\cap \Delta_{n-1}\right)}{\mathcal{H}^{n-1} \left(\Delta_{n-1}\right)}\] by $\mathrm{Unif}\Delta_{n-1}$. If the is no confusion to arise, we will use the same symbol to denote a random vector, the probability distribution of which is $\mathrm{Unif}\Delta_{n-1}$.  
\section{LLT for exponential variables}
\begin{lemma}\label{sympol}
	For $\tau_1,\ldots,\tau_n\in \R$, let $\sigma_j(\tau):= \sum_{1\leq i_1<\ldots<i_j\leq n} \tau_{i_1}\cdot\ldots\cdot\tau_{i_j}$ be the $j$-th elementary symmetric polynomial and $\pi_j(\tau):= \sum_k \tau_k^j$. Then 
	\[
		\sigma_j = \frac{1}{j!} \pi_1^j + U_j\left(\pi_1,\ldots,\pi_j\right)
	\]
	for some polynomial $U_j$ in $j$ variables, homogeneous of degree $j$ in $\tau_1,\ldots,\tau_n$ that does not contain a monomial $\pi_1^j$. 
\end{lemma}

\begin{proof}
	By induction. For $j=1$, $\sigma_1=\pi_1$. By Newton's identities \cite[p.23]{zbMATH00739282},
	\[
	j\sigma_j= \sum_{i=1}^j (-1)^{i+1} \sigma_{j-i}\pi_i= \sum_{i=1}^j(-1)^{i+1}\left(c_{j-i} \pi_1^{j-i}+ U_{j-i} \right)\pi_i.
	\]
	Each summand is of degree $j-i+i=j$ in $\tau$. One can get pure $\pi_1^j$ from the summand for $i=1$ and it appears with coefficient $(-1)^{1+1}\frac{1}{(j-1)!}$; all the other ones contain at least one occurrence of $\pi_2,\ldots,\pi_n$.  
\end{proof}
	
\begin{lemma}\label{khintail}
	Let $\ell,p\in \N_0$. For $M\geq 1$ and sufficiently large $n$ and sufficiently small $m$,
	\[
		\iint_{|\xi|>M}|\xi|^\ell \prod_{k=1}^n \left(1+t_k^2\right)^{-\frac{1}{2}}d \xi \lesssim_{\ell,p} M^{-p}.
	\]
\end{lemma}	

\begin{proof}
	Denote $\tau_k=t_k^2\geq 0$. We callously bound $\prod_k \left(1+\tau_k\right)\geq 1+\sigma_j(\tau_1,\ldots,\tau_n)$. Here, $\pi_1=\sum t_k^2= |\xi|^2$, and for $i\geq 2$, by $\|\cdot\|_{\ell^{2i}}\leq \|\cdot\|_{\ell^2}$,
	\[
	\pi_i \leq \sum |\xi|^{2i} |v_k|^{2i}\leq |\xi|^{2i} \left(\sum \left|v_k\right|^3\right)^\frac{2i}{3} = |\xi|^{2i} m^{2i}\leq |\xi|^{2i} m^4
	\]
	Therefore, in the expansion of $\sigma_j$ shown in Lemma \ref{sympol}, each monomial in $U_j$ is bounded by $|\xi|^{2j} \cdot m^4$, as it contains at least one instance of $\pi_i$ for $i\geq 2$. Thus, 
	\[
	\sigma_j\geq |\xi|^{2j}\left(\frac{1}{j!}-C_jm^4\right)\gtrsim_{j}|\xi|^{2j},
	\]
    where the constant depends on the upper bound for $m$ for a given $j$, e.g. $\frac{1}{2}$ for $m^4\leq \frac{1}{2C_j j!}$. Ultimately, we take e.g. $j=\ell+2+p$ and get 
	\[
	\iint_{|\xi|>M}|\xi|^\ell \prod_{k=1}^n \left(1+t_k^2\right)^{-\frac{1}{2}}d\xi \lesssim_{\ell,j} \iint_{ |\xi|>M }|\xi|^\ell \left(1+|\xi|^{2j}\right)^{-\frac{1}{2}}d \xi\eqsim 2\pi \int_{M}^\infty \frac{r^{\ell+1}}{1+r^{j}}dr
	\]
	\[
	\eqsim \int_{M}^\infty r^{\ell+1-j} dr \eqsim_p  M^{\ell+2-j} = M^{-p}.
	\]
	
	\end{proof}

\begin{lemma}\label{tworeg}
	For any $\ell\in \N_0$, sufficiently large $n$ and sufficiently small $m$, 
	\[
	\iint_{\R^2}|\xi|^\ell \left(e^F-e^H\right) d \xi\lesssim_{\ell} m^4
	\]	
\end{lemma}
	
	\begin{proof}
		We shall divide $\R^2$ into two regions. On the $|\xi|>\frac{1}{m}$ region, the bound is just quoting Lemma \ref{khintail} for $p=4$, since $F\geq H$ and $e^F= \prod\left(1+t_k^2\right)^{-\frac{1}{2}}$. Whenever $|\xi|\leq \frac{1}{m}$, we have
		\[t_k^2\leq |\xi|^2\cdot |v_k|^2\leq |\xi|^2 \max |v_k|^2\leq|\xi|^2 m^2\leq  1,\]
		thus 
		\[\sum t_k^2-\ln\left(1+t_k^2\right)\lesssim \sum t_k^4 \leq |\xi|^4 m^4\leq 1.\]
		 By the inequality $e^z-1\lesssim z$ for $0\leq z\leq \frac{1}{2}$,  
		\[
		\iint_{|\xi|\leq \frac{1}{m}} \left(e^F-e^H\right) |\xi|^\ell d\xi = \iint_{|\xi|\leq \frac{1}{m}} e^{H} \left(e^{F-H}-1\right) |\xi|^\ell d\xi
		\]
		\[
		= \iint_{|\xi|\leq \frac{1}{m}} e^{H} \left(e^{\frac{1}{2}\sum \left(t_k^2- \ln\left(1+t_k^2\right)\right) }-1\right) |\xi|^\ell d\xi
		\]
		\[
		\lesssim  \iint_{|\xi|\leq \frac{1}{m}} e^{H} \left(\frac{1}{2}\sum \left(t_k^2- \ln\left(1+t_k^2\right)\right) \right) |\xi|^\ell d\xi
		\]
		\[\lesssim  \iint_{|\xi|\leq \frac{1}{m}} e^{H} |\xi|^{\ell+4}m^4 d\xi\leq m^4 \iint_{\R^2} e^{-|\xi|^2/2} |\xi|^{\ell+4} d\xi\lesssim_{\ell} m^4.\]
	\end{proof}

\begin{corollary}\label{ezminushpolyl1}
 		For any $\ell\in \N_0$, sufficiently large $n$ and sufficiently small $m$, 
 	\[
 	\iint_{\R^2}|\xi|^\ell \left|e^{F+iG}-e^H\right| d \xi\lesssim m^3
 	\]
\end{corollary}

\begin{proof}
		By the inequality $t-\arctan t\leq \frac{1}{3}t^3$ for $t>0$ and $\left|t_k\right|\leq \left|v_k\right|\cdot |\xi|$,
		\[\left|e^{iG}-1\right|\leq |G|	\leq  \sum\left|t_k\right|-\arctan |t_k|\leq \frac{1}{3}\sum |t_k|^3\lesssim |\xi|^3 m^3.
		\]
		Using this bound, we upgrade the Lemma \ref{tworeg}. 
		\[\left|e^{F+iG}-e^H\right| \leq \left|\left(e^F-e^H\right)e^{iG}\right|+\left|e^H\left(e^{iG}-1\right)\right|,\]
		so 
		\[
		\iint_{\R^2}|\xi|^\ell \left|e^{F+iG}-e^H\right| d \xi\leq \iint_{\R^2}|\xi|^\ell \left|e^{F}-e^H\right| d \xi+ \iint_{\R^2}|\xi|^\ell e^H \left|e^{iG}-1\right| d \xi\lesssim_{\ell}  
		\]
		\[m^4+ m^3\iint_{\R^2}|\xi|^{\ell+3} e^H  d \xi\lesssim_{\ell} m^3.\]
\end{proof}

\begin{lemma} \label{glemma}For any $r\geq 1$ and $ b\in \{1,2\}$, there is a polynomial $P_r$ such that $\left|\partial_b^r G\right|\leq P_r(|\xi|)m^3$. 
\end{lemma}
\begin{proof}
	%Calculations, not unlike chess, speak for themselves. 
 This is a direct computation: 
	\[ -\d_b G = \sum_k \frac{d t_k}{d \xi_b}\cdot \frac{d}{dt_k}\left(t_k- \arctan t_k \right)	= \sum v_{kb}\left(1-\frac{1}{1+t_k^2}\right) = \sum v_{kb} \frac{t_k^2}{1+t_k^2},\]
	thus 
	\[\left|\d_b G\right|\leq \sum \left|v_k\right| \cdot t_k^2\leq  \sum \left|v_k\right|^3 |\xi|^2= m^3 |\xi|^2.
	\]
	Differentiating further, we get 
	\[-\d_b^2 G= \sum v_{kb}^2 \frac{d}{dt_k}\frac{1}{1+t_k^2} = \sum v_{kb}^2\frac{2t_k}{\left(1+t_k^2\right)^2},\]
	so 
	\[\left|\d_b^2 G\right|\leq \sum \left|v_k\right|^2 \left|t_k\right| \leq \sum \left|v_k\right|^3 |\xi|=m^3 |\xi|.\]
	It is easy to see by induction that $\left(\frac{d}{dt}\right)^r \left(t-\arctan t\right)=  \frac{P_r(t)}{\left(1+t^2\right)^{2^r}}$, so for $r\geq 3$,
	\[\left|\d_b^r G\right|= \left|\sum v_{kb}^r \left(\frac{d}{dt_k}\right)^r \left(t_k-\arctan t_k\right) \right|\]
	\[\leq \sum \left|v_{kb}\right|^r \left| \frac{P_r(t_k)}{\left(1+t_k^2\right)^{2^r}} \right|\leq \sum \left|v_k\right|^r P_r\left(|v_k|\cdot |\xi|\right)\leq m^3 P_r(|\xi|).\]
	
\end{proof}
 
\begin{lemma}\label{fhlemma}
	For any $r\geq 1$ and $ b\in \{1,2\}$, $\left|\d_b^r(F-H)\right|\leq m^3 P_r(|\xi|)$. 
	\end{lemma}
	
	\begin{proof}
	Obviously, $\d_b H= -\xi_b$, $\d^2 H=-1$, $\d^3 H=0$. It will be useful to remember that since $\sum x_k=0$ and $\sum x_k^2=1$,
	\[
	\sum n^{-1/2}t_k= n^{-1/2}\sum \left(\xi_1 x_k +\xi_2 n^{-1/2}\right)= \xi_2
	\]
	and
	\[
	\sum x_k t_k = \xi_1 \sum x_k^2 + \xi_2 n^{-1/2}\sum x_k = \xi_1,
	\]
	which can be compressed into 
	\[\sum_k v_{kb} t_k = \xi_b,\]
	which in turn was obvious from the fact that $\sum t_k v_k = \sum \ilsk {\xi}{v_k}v_k =V^TV\xi = \xi$.
 %, but \sout{I am a little dim in the headlights} a more direct computation looks better. 
 Moreover, for the same reasons,
	\[\sum_k v_{kb}^2=1.\]\par
	%Minding the fact that this one of the rare parts of this reasoning in which constants and signs actually matter,
        Therefore, for both $b\in\{1,2\}$,
	\[
	\d_b F= \d_b \sum -\frac{1}{2}\ln\left(1+t_k^2\right)= -\sum v_{kb}\frac{t_k}{1+t_k^2},
	\]
	\[
	\d_b^2 F= -\sum v_{kb}^2\frac{1+t_k^2-2t_k^2}{\left(1+t_k^2\right)^2}= -\sum v_{kb}^2\frac{1-t_k^2}{\left(1+t_k^2\right)^2},
	\]
	and for $r\geq 3$, 
	\[
	\d_b^r F=  \sum v_{kb}^r \frac{P_r(t_k)}{\left(1+t_k^2\right)^{2^r}}.
	\]
	We are now ready to estimate the differences.	
	\[
	\left|\d_b(F- H)\right|= \left|\xi_b-\sum v_{kb}\frac{t_k}{1+t_k^2}\right|= \left|\sum v_{kb}t_k-\sum v_{kb}\frac{t_k}{1+t_k^2}\right|
	\]
	\[
	=\left|\sum v_{kb}\frac{t_k^3}{1+t_k^2}\right|\leq \sum \left|v_k\right| |t_k|^3\leq \sum \left|v_k\right|^4\cdot |\xi|^3\leq  m^{3}|\xi|^3,
	\]
	\[
	\left|\d_b^2 (F-H)\right| = \left|1-\sum v_{kb}^2\frac{1-t_k^2}{\left(1+t_k^2\right)^2}\right|\]
	\[=\left|\sum v_{kb}^2-\sum v_{kb}^2 \frac{1-t_k^2}{\left(1+t_k^2\right)^2}\right|= \left|\sum v_{kb}^2 \frac{t_k^2-1+1+2t_k^2+t_k^4}{\left(1+t_k^2\right)^2}\right|
	\]
	\[
	= \left|\sum v_{kb}^2\frac{3t_k^2+t_k^4}{\left(1+t_k^2\right)^2} \right|\lesssim \sum \left|v_k\right|^2 \left(\left|t_k\right|^2 + \left|t_k\right|^4\right)\leq m^4 |\xi|^2 + m^6 |\xi|^4\leq m^3 P_2(|\xi|) ,
	\]
	and for $r\geq 3$,
	\[
	\left|\d_b^r (F - H)\right|= \left|\d_b^3 F\right|\leq \sum \left|v_k\right|^r P_r\left(\left|t_k\right|\right)\leq m^3 P_r(|\xi|).
	\]
	\end{proof}

\begin{theorem}\label{mainl1poly}
 For $b\in \{1,2\}$, $\ell,p\in \N_0$ and sufficiently small $m$,
 \[
 \iint_{\R^2}|\xi|^\ell \left|\d_b^p \left(e^{Z}-e^H\right)\right|d \xi\lesssim_{\ell,p} m^3.
 \]	
\end{theorem}
	\begin{proof}
%	Firstly, we make $Z-H$ the main character. 
	By Leibniz's product rule,
	\[\d_b^p \left(e^{Z}-e^H\right)= \d_b^p e^H\left(e^{Z-H}-1\right)=\sum_{q=0}^p \binom{p}{q}\d_b^{p-q}e^H \d_b^q\left(e^{Z-H}-1\right) \]
	\[= \sum_{q=0}^p (-1)^{p-q}\binom{p}{q}\he{p-q}{\xi_b}e^H \d_b^q\left(e^{Z-H}-1\right).\]
	Thus
	\[\left|\d_b^p \left(e^{Z}-e^H\right)\right|\lesssim_p \left|\he{p}{\xi_b}e^H\left(e^{Z-H}-1\right)\right|+  \sum_{q=1}^p\left| \he{p-q}{\xi_b}e^H\d_b^q e^{Z-H}\right|.
	\]
	Our inequality for the first summand follows directly from Corollary \ref{ezminushpolyl1}:
	\[
	\iint_{\R^2}|\xi|^\ell \left|\he{p}{\xi_b}e^H\left(e^{Z-H}-1\right)\right| d\xi\leq \iint_{\R^2}P_{\ell,p}(|\xi|)\left|e^{Z}-e^H\right| d\xi \lesssim_{\ell,p} m^3.
	\]
	For the second one, we have to use the Fa\`{a} di Bruno's Formula:
 \[
\d^q \left(f\circ g\right)= \sum_{s_1+2s_2+\ldots+qs_q=q}\frac{q!}{\prod_{j=1}^q s_j! j!^{s_j}}\d^{s_1+\ldots+s_q}f \prod_{j=1}^q \left(\d^j g\right)^{s_j}.\]
  Ignoring the constants,
	\[
	\left|\d_b^q e^{Z-H}\right|\lesssim_q \sum_{s_1+2s_2+\ldots+qs_q=q}e^{Z-H}\prod_{j=1}^q\left(\d_b^j(Z-H)\right)^{s_j}.
	\]
	Since all derivatives are of order $\geq 1$, we may use Lemma \ref{glemma} + Lemma \ref{fhlemma}:
	\[
	\left|\d_b^q e^{Z-H}\right|\lesssim m^3 e^{Z-H} P_q(|\xi|)\]
	for $q\geq 1$. Ultimately,
	\[\iint_{\R^2} |\xi|^\ell\sum_{q=1}^p\left| \he{p-q}{\xi_b}e^H\d_b^q e^{Z-H}\right|d \xi
	\]
	\[
	\lesssim_p \iint_{\R^2} |\xi|^\ell \sum_q P_q(\xi) m^3 e^Z d\xi= m^3 \iint_{\R^2}P_{\ell,p}(|\xi|)e^Z d\xi \]
	\[= m^3 \iint_{\R^2}P_{\ell,p}(|\xi|)e^H d\xi + m^3 \iint_{\R^2}P_{\ell,p}(|\xi|)\left(e^Z-e^H\right) d\xi\lesssim_{\ell,p} m^3.
	\]
	\end{proof}

\section{Estimates of PDF for quotients}
 
We will make use of the following folklore lemma, expressing the probability density function (PDF)
%\footnote{Portable Docment Format} 
of a quotient of random variables in terms of their joint PDF.
%\footnote{Planetary Defense Force}.
\begin{lemma} \label{pdfquotient}
	\[
	\mathrm{PDF}_{X_1/X_2}(s)=\int_\R |y|\mathrm{PDF}_{(X_1,X_2)}(sy,y)d y.
	\]

\end{lemma}

\begin{proof} For any $t\in\R$,
\[ \int_{-\infty}^t \mathrm{RHS}(s)ds= \int_\R |y|\int_{\infty}^t \mathrm{PDF}_{(X_1,X_2)}(sy,y)dsdy
\]
\[=\int_0^\infty y\int_{-\infty}^t \mathrm{PDF}_{(X_1,X_2)}(sy,y)dsdy - \int_{-\infty}^0 y\int_{-\infty}^t \mathrm{PDF}_{(X_1,X_2)}(sy,y)dsdy
\]
\[
=\int_0^\infty\int_{-\infty}^{ty} \mathrm{PDF}_{(X_1,X_2)}(w,y)dwdy + \int_{-\infty}^0\int_{ty}^\infty \mathrm{PDF}_{(X_1,X_2)}(w,y)dwdy
\]
\[
=\iint_{\left\{(w,y): y>0\& w<ty\text{ or } y<0 \& w>ty\right\}}\mathrm{PDF}_{(X_1,X_2)}(w,y)dwdy\]
\[= \iint_{\left\{(w,y): \frac{w}{y}<t\right\}}\mathrm{PDF}_{(X_1,X_2)}(w,y)dwdy= \mathbb{P}\left(\frac{X_1}{X_2}<t\right)= \int_{-\infty}^t \mathrm{PDF}_{X_1/X_2}(s) ds.
\]

\end{proof}

\begin{theorem}\label{LLTvec}
	Let $P_1,\ldots,P_n$ be independent and $\mathrm{Exp}(1)$-distributed. Define
	\[Q=\left(Q_1,Q_2\right)=\left( \sum x_k(P_k-1), n^{-1/2}\sum (P_k-1)\right)\]
	 and let $N=(N_1,N_2)$ be $\mathrm{N}(0,I_2)$-distributed. Then for any $p,q,r\in\N_0$ and $s_1,s_2\in \R $,
 	\[	\left(\left|s_1\right|^r+\left|s_2\right|^r\right)\left|\d_{s_1}^p\d_{s_2}^q \left(\mathrm{PDF}_Q-\mathrm{PDF}_N\right)(s_1,s_2)\right|\lesssim_{p,q,r} m^3.\] 
\end{theorem}

\begin{proof}
	Firstly, we calculate the characteristic function of $$Q= \sum (x_k,n^{-1/2})(P_k-1).$$ Since 
	\[\varphi_{\mathrm{Exp}(1)-1}(t)= \E e^{it(P_1-1)}= \int_0^\infty e^{it(p-1)}\cdot e^{-p}dp= e^{-it}\int_0^\infty e^{-p(1-it)}dp =\frac{e^{-it}}{1-it},\]
	we have
	\[
	\varphi_Q(\xi)= \E e^{i\ilsk{\xi }{Q}}= \E e^{i \sum \ilsk{\xi}{v_k}(P_k-1)}=\prod \E e^{it_k (P_k-1)}= \prod \frac{e^{-it_k}}{1-it_k}
	\]
	\[
	=\prod \frac{e^{-it_k}}{\left|1-it_k\right|\cdot e^{-i\arctan t_k}}=\prod \frac{e^{-i(t_k-\arctan t_k)}}{\sqrt{1+t_k^2}}
	\]
	\[= e^{-\frac{1}{2}\sum\ln\left(1+t_k^2\right)-i\sum(t_k-\arctan t_k)}= e^{F+iG}.
	\]
	Therefore, by Fourier inversion formula,
	\[
	(2\pi)^2\left|s_1^r\partial_{s_1}^p\partial_{s_2}^q\left(\mathrm{PDF}_Q-\mathrm{PDF}_N\right)\left(s_1,s_2\right)\right|= \left|s_1^r\partial_{s_1}^p\partial_{s_2}^q\iint_{\R^2}e^{-i\ilsk{s}{\xi}}\left(\varphi_Q-\varphi_N\right)(\xi)d \xi\right|
	\]
	\[
	=\left|s_1^r\iint_{\R^2} \left(\varphi_Q-\varphi_N\right)(\xi)\partial_{s_1}^p\partial_{s_2}^q e^{-i\ilsk{s}{\xi}}d \xi\right| = 
	\left| s_1^r\iint_{\R^2}  \left(\varphi_Q-\varphi_N\right)(\xi) \xi_1^p\xi_2^q e^{-i\ilsk{s}{\xi}} d \xi \right|
\]          
	\[= \left| \iint_{\R^2}  \left(\varphi_Q-\varphi_N\right)(\xi) \xi_1^p\xi_2^q \partial_{\xi_1}^r e^{-i\ilsk{s}{\xi}} d \xi \right|
= \left| \iint_{\R^2}  e^{-i\ilsk{s}{\xi}} \partial_{\xi_1}^r  \left( \left(\varphi_Q-\varphi_N\right)(\xi) \xi_1^p\xi_2^q \right) d \xi \right|
	\]
 \[
=\left| \iint_{\R^2}  e^{-i\ilsk{s}{\xi}} \sum_{j=0}^r\binom{r}{j}  \left( \d_{\xi_1}^{r-j} \left(\varphi_Q-\varphi_N\right)(\xi) \right)\left( \d_{\xi_1}^j\xi_1^p\xi_2^q \right) d \xi \right|
 \]
	\[\lesssim_{p,r} \sum_{j=0}^{\min(r,p)} \iint_{\R^2}  \left| \xi_2^q \xi_1^{p-j} \partial_{\xi_1}^{r-j}\left(\varphi_Q-\varphi_N\right)(\xi) \right|d \xi \]
\[
\leq \sum_{j=0}^{\min(r,p)} \iint_{\R^2}  |\xi|^{p+q-j}\left|\partial_{\xi_1}^{r-j}\left(\varphi_Q-\varphi_N\right)(\xi) \right|d \xi 
\lesssim_{p,q,r} m^3\]
by Theorem \ref{mainl1poly}. The identical calculation is valid for multiplying by $s_2^r$. 
\end{proof}

\begin{theorem} \label{quotqcloseton}
Under the assumptions of Theorem \ref{LLTvec},
\[\left|s^p\partial_s^q\left(\mathrm{PDF}_{\frac{Q_1}{1+n^{-1/2}Q_2}}-\mathrm{PDF}_{\frac{N_1}{1+n^{-1/2}N_2}}\right)(s) \right| \lesssim_{p,q} m^3\]  for $p,q\in \N_0$ and $|s|\leq n$.
\end{theorem}

\begin{proof}
	By Lemma \ref{pdfquotient},
	\[
	\left|s^p\partial_s^q\left(\mathrm{PDF}_{\frac{Q_1}{1+n^{-1/2}Q_2}}-\mathrm{PDF}_{\frac{N_1}{1+n^{-1/2}N_2}}\right)(s)\right|\]

\[= \left| s^p\partial_s^q\int_\R |y| \left(\mathrm{PDF}_{(Q_1,1+n^{-1/2}Q_2)}-\mathrm{PDF}_{(N_1,1+n^{-1/2}N_2)}\right)(sy,y)d y\right|	\]
	
\[
	=\left|s^p\partial_s^q \int_\R |y| n^{1/2}\left(\mathrm{PDF}_{(Q_1,Q_2)}-\mathrm{PDF}_{(N_1,N_2)}\right)(sy,n^{1/2}(y-1))d y\right|\]

\[=\left|s^p \int_\R |y| n^{1/2}y^q \left[\partial^{(q,0)}\left(\mathrm{PDF}_{(Q_1,Q_2)}-\mathrm{PDF}_{(N_1,N_2)}\right)\right](sy,n^{1/2}(y-1))d y\right|
\]

\[\label{uglyint1}\lesssim_{A,B} m^3\int_\R \frac{\left|s^p y^{q+1}n^{1/2}\right|}{1+|sy|^A+\left|n^{1/2}(y-1)\right|^B} d y,
\]
where the last inequality is an application of Theorem \ref{LLTvec} for $\partial_{s_1}^q$ and $r\in\{0,A,B\}$. \par
Now we have to choose the appropriate $A,B$ to prove that the integral is bounded by a constant that depends only on $p,q$. Without loss of generality, $s>0$. \par 
For $s\leq 2$, by substituting $z=n^{1/2}(y-1)$, 
\[ \int_\R \frac{\left|s^p y^{q+1}n^{1/2}\right|}{1+|sy|^A+\left|n^{1/2}(y-1)\right|^B} dy\lesssim_{p} \int_\R \frac{|y|^{q+1}n^{1/2}}{1+\left|n^{1/2}(y-1)\right|^B}dy \]
\[
=\int_\R \frac{\left|1+n^{-1/2}z\right|^{q+1}}{1+\left|z\right|^B}dz \lesssim_q \int_\R \frac{(1+|z|)^{q+1}}{1+|z|^B}\lesssim_q 1
\]
provided that $B\geq q+3$. \par
From now on, $s>2$. We will split the intergal in two parts: $|y|\leq \frac{1}{2}$ and $|y|>\frac{1}{2}$. For any $t_1,t_2\geq 0$ and $A\geq p$, $B\geq 2(q+3)$,
\[t_1^p\left(1+t_2^{q+3}\right)=t_1^p + t_1^p t_2^{q+3}\lesssim t_1^p+t_1^{2p} + t_2^{2(q+3)}\lesssim 1+t_1^A+t_2^B.\]
Therefore,
\[
\int_{\R-\left[-\frac{1}{2},\frac{1}{2}\right]} \frac{\left|s^p y^{q+1}n^{1/2}\right|}{1+|sy|^A+\left|n^{1/2}(y-1)\right|^B} dy 
\lesssim \int_{\R-\left[-\frac{1}{2},\frac{1}{2}\right]} \frac{\left|s^p y^{q+1}n^{1/2}\right|}{|sy|^A\left(1+\left|n^{1/2}(y-1)\right|^B\right)} dy 
\]
\[
\leq \int_{\R-\left[-\frac{1}{2},\frac{1}{2}\right]} \frac{\left|s^p y^{q+1}n^{1/2}\right|}{(s/2)^A\left(1+\left|n^{1/2}(y-1)\right|^B\right)} dy \leq 
2^A s^{p-A}\int_\R \frac{|y|^{q+1}n^{1/2}}{1+\left|n^{1/2}(y-1)\right|^B}dy \lesssim_{q,A} 1
\]
by the same calculation as for $s\leq 2$, provided that $A\geq p$ and $B\geq 2(q+3)$. \par 
For $|y|\leq \frac{1}{2}$, we have $|1-y|\geq \frac{1}{2}$. Moreover, the integral in question is increasing in $p$ due to $s>1$, so without loss of generality $p\geq q+2$. Thus
\[ \int_{\left[-\frac{1}{2},\frac{1}{2}\right]} \frac{\left|s^p y^{q+1}n^{1/2}\right|}{1+|sy|^A+\left|n^{1/2}(y-1)\right|^B} dy \eqsim_B \int_{\left[-\frac{1}{2},\frac{1}{2}\right]} \frac{\left|s^p y^{q+1}n^{1/2}\right|}{|sy|^A+ n^{B/2}} dy 
\]
\[=2n^{1/2} s^{p-q-1} \int_0^\frac{1}{2} \frac{(sy)^{q+1}}{|sy|^A+ n^{B/2}} dy
\eqsim n^{1/2} s^{p-q-2} \int_0^\frac{s}{2} \frac{z^{q+1}}{z^A+ n^{B/2}} dz\]
\[ =n^{1/2} s^{p-q-2} \int_0^\frac{s/2}{n^{\frac{B}{2A}}w} \frac{\left(n^{\frac{B}{2A}}w\right)^{q+1}}{\left(n^{\frac{B}{2A}}\right)^A+ n^{B/2}} dn^{\frac{B}{2A}}w
\]
\[=n^{1/2}s^{p-q-2}\int_0^{\frac{1}{2}s n^{-\frac{B}{2A}}} \frac{n^{\frac{(q+1)B}{2A}}w^{q+1}}{n^{B/2}\left(1+w^A\right)}dw\]
\[=n^{\frac{1}{2}+\frac{B}{2}\left(\frac{q+1}{A}-1\right)} s^{p-q-2} \int_0^{\frac{1}{2}s n^{-\frac{B}{2A}}} \frac{w^{q+1}}{1+w^A}dw \]
\[
\leq n^{\frac{1}{2}+\frac{B}{2}\left(\frac{q+1}{A}-1\right)+p-q-2} \int_0^\infty \frac{w^{q+1}}{1+w^A}dw. 
\]
This is made bounded e.g. by taking first $A=2(q+3)$ and then $B\geq 4p$. 
\end{proof}

\begin{theorem} \label{puregauss} The inequality 
	\[\left|t^p\partial_t^q\left(\mathrm{PDF}_{N_1}-\mathrm{PDF}_{\frac{N_1}{1+n^{-1/2}N_2}}\right)(t) \right| \lesssim_{p,q} n^{-1/2}\]
	holds\footnote{Probably one can take a better exponent on the RHS, but
 %\sout{we are lazy} 
 other bounds are already not better than $n^{-1/2}$.} for $p,q\in \N_0$ and $|t|\leq n$.  
\end{theorem}

\begin{proof}
This is yet another application of Lemma \ref{pdfquotient}. Denote $\rho=\mathrm{PDF}_{N_1}$ and $h_k(t)=\d_t^k \rho(t)= (2\pi)^{-1/2}(-1)^k \he{k}{t} e^{-t^2/2}$. We have
\[
\d_t^q\mathrm{PDF}_{\frac{N_1}{1+n^{-1/2}N_2}}(t)= \d_t^q \int_\R |y|\mathrm{PDF}_{\left(N_1,1+n^{-1/2}N_2\right)}(yt,y)dy
\]
\[
=\d_t^q \int_\R n^{1/2}|y|\mathrm{PDF}_{\left(N_1,N_2\right)}(yt,n^{1/2}(y-1))dy= \d_t^q \int_\R n^{-1/2}|y|(yt)\left(n^{-1/2}(y-1) \right)dy
\]

\[
=\int_\R n^{1/2}|y|y^q h_q(yt)\left(n^{1/2}(y-1) \right)dy.
\]
This is supposed to be $\lesssim t^{-p}n^{-1/2}$-close to $h_q(t)$, which is to be expected, since $n^{1/2}\left(n^{1/2}(y-1) \right)$ is an approximation kernel around $y=1$. We will prove
\begin{equation}\left|\int_\R n^{1/2}\left(|y|y^q-1\right) h_q(yt)\left(n^{1/2}(y-1) \right)dy\right|\lesssim_{p,q} t^{-p}n^{-1/2} \label{michbound}
\end{equation}
and
\begin{equation}\label{ichbound} \left|\int_\R \left(h_q(t)-h_q(yt)\right)n^{1/2}\left(n^{-1/2}(y-1)\right)dy \right|\lesssim_{p,q} t^{-p}n^{-1/2}.
	\end{equation}
	Without loss of generality, $t>0$, because $h_q$ is either an odd or an even function. \par 
	In order to prove \eqref{michbound}, we notice that 
	\[
	\left||y|y^q-1\right|\simeq_q \begin{cases} 1+|y|^{q+1}\simeq_q |y-1|^{q+1}\text{ for }y<0\\ |y-1| \text{ for } y\in (0,1)\\ |y-1|^{q+1}\text{ for }y>1.\end{cases}
	\]
	Thus, the first integral is bounded by
	\[
	\int_\R n^{1/2}\left(|y-1|+|y-1|^{q+1}\right)h_q(yt)\left(n^{1/2}(y-1) \right)dy
	\]
	\[
	=\int_\R \left(|n^{-1/2}z|+|n^{-1/2}z|^{q+1}\right)h_q\left(t\left(1+n^{-1/2}z\right)\right)(z) dz
	\]
	\[
	\leq n^{-1/2}\int_\R P_q(z) h_q\left(t\left(1+n^{-1/2}z\right)\right)(z) dz\] \[\leq n^{-1/2}\tilde{P}_q(t)\int_\R P_q(z)\left(t\left(1+n^{-1/2}z\right)\right)(z)dz.
	\]
 because $h_q(yt)=P_q(yt)\rho(yt)$. 
	Since we have to prove the inequality for arbitrary $p$ anyway, it is enough to show
\[\int_\R |z|^A \left(t\left(1+n^{-1/2}z\right) \right)\rho\left(z\right)dz\lesssim_{A,p} t^{-p}\]
for any $A,p\in \N_0$. We may assume $t\gtrsim 1$, because otherwise we simply bound $\rho\left(t\left(1+n^{-1/2}z\right) \right)$ by a constant and RHS is $\gtrsim 1$. For a fixed $A$, the function $\left(\frac{1}{2}n^{1/2},\infty \right)\ni w\mapsto w^A \rho(w)$ is decreasing, so for  $z<-\frac{1}{2}n^{1/2}$, 
\[n^{1/2} t^{p} |z|^A \rho(z)\lesssim n^{p+\frac{1}{2}} n^{A/2} e^{-n/8}\lesssim_{A,p} 1.\]
Therefore 
\[
\int_{-\infty}^{-\frac{1}{2}n^{1/2}} |z|^A \rho\left(t\left(1+n^{-1/2}z\right) \right)\rho\left(z\right)dz \lesssim_{A,p}
n^{-1/2}t^{-p}\int_{-\infty}^{-\frac{1}{2}n^{1/2}} \rho\left(t\left(1+n^{-1/2}z\right) \right)dz 
\]
\[
=n^{-1/2}t^{-p}\int_{-\infty}^{\frac{1}{2}}\rho(ty)n^{1/2}dy\leq  t^{-p} \int_\R \rho(ty)dy =t^{-p-1}\lesssim t^{-p}
\]
as intended. One the other hand, if $z>-\frac{1}{2}n^{1/2}$ then $1+n^{-1/2}z>\frac{1}{2}$, so 
\[
t^p \int_{-\frac{1}{2}n^{1/2}}^\infty |z|^A \rho\left(t\left(1+n^{-1/2}z\right) \right)\rho\left(z\right)dz
\leq t^p \int_{-\frac{1}{2}n^{1/2}}^\infty |z|^A \rho\left(t/2\right)\rho\left(z\right)dz
\]
\[\leq t^{p}\rho(t/2)\int_\R |z|^A\rho(z)dz\lesssim_{A,p} 1\] 
which concludes the proof of \eqref{michbound}.\par 
For \eqref{ichbound}, we use a brutal Mean Value Theorem bound, minding that $[yt,t]$ denotes $\mathrm{conv}\{yt,t\}$ even if $yt>t$. 
\[
n^{1/2}t^p\left|\int_\R \left(h_q(t)-h_q(yt)\right)n^{1/2}\rho\left(n^{-1/2}(y-1)\right)dy \right|
\]
\[
\leq n^{1/2}t^p\int_\R \left|h_q(t)-h_q(yt)\right|n^{1/2}\rho\left(n^{-1/2}(y-1)\right)dy
\]
\[
\leq n^{1/2}t^p\int_\R t\left|y-1\right| \sup_{[yt,t]}\left|h_{q+1}\right|n^{1/2}\rho\left(n^{-1/2}(y-1)\right)dy
\]
\[
=t^{p+1}\int_\R|z|\rho(z) \sup_{[\left(1+n^{-1/2}z\right)t,t]}\left|h_{q+1}\right|dz.
\]
For $t\lesssim 1$ this is just bounded by a constant since $|h_{q+1}|$ is bounded. Take $t\gtrsim 1$. We again split the integral in two. For $z>-\frac{1}{2}n^{-1/2}$ and $t$ far enough (dependent on $q$) from $1$, $[\left(1+n^{-1/2}z\right)t,t]\subseteq [t/2,t]$ and $|h_{q+1}|$ is decreasing on this interval, so 
\[
t^{p+1}\int_{-\frac{1}{2}n^{1/2}}^\infty |z|\rho(z) \sup_{[\left(1+n^{-1/2}z\right)t,t]}\left|h_{q+1}\right|dz \leq t^{p+1}\left|h_{q+1}(t/2)\right| \int_\R \rho(z)|z|dz\lesssim_{p,q} 1.
\]
On the other hand, for a given $A\in \N_0$ and big enough $w$, 
\[
\int_w^\infty \rho(z)z^Adz = \frac{1}{\sqrt{2\pi}}\int_w^\infty e^{-\left(\frac{z^2}{2}-A\ln z\right)}dz\leq \int_w^\infty e^{-z}dz= e^{-w},
\]
so 
\[
t^{p+1}\int_{-\infty}^{-\frac{1}{2}n^{1/2}} |z|\rho(z) \sup_{[\left(1+n^{-1/2}z\right)t,t]}\left|h_{q+1}\right|dz \leq t^{p+1}\sup_\R \left|h_{q+1}\right| \int_{-\infty}^{-\frac{1}{2}n^{1/2}} |z|\rho(z)dz
\]
\[
\lesssim_q t^{p+1}e^{-\frac{1}{2}n^{1/2}}\leq n^{p+1}e^{-\frac{1}{2}n^{1/2}}\lesssim_p 1. 
\]
\end{proof}

\section{Proof of Theorem 1}

The following lemma is a well known property of the so-called flat Dirichlet distribution (cf. \cite{zbMATH03954145})

\begin{lemma}\label{unifsimplex}
If $P_1,\ldots,P_n$ are iid $\mathrm{Exp(1)}$, then the distribution of the random vector 
\[\frac{\left(P_1,\ldots,P_n\right)}{\sum_{k=1}^n P_k}\]
is $\mathrm{Unif}\Delta_{n-1}$. \end{lemma}

%\begin{theorem} \label{mainthm}If $p,q\in \N_0$, $\sum x_k=0$, $\sum x_k^2=1$, $\sum \left|x_k\right|^3$ is small enough and\footnote{This is not a very restrictive assumption since the sum is $0$ unless $\frac{|t|}{n}\leq \max |x_k|\leq 1$.} $|t|\leq n$, then
%\[
%\left|t^p \left(\frac{(-1)^q}{\sqrt{2\pi}}\he{q}{t}e^{-t^2/2} - $$\d_t^q\sum_{k=1}^n \frac{\max\left(x_k-\frac{t}{n},0\right)^{n-2} }{\prod_{j\neq k} \left(x_k-x_j\right)}
%  \right)  \right| \lesssim \sum_{k=1}^n \left|x_k\right|^3.\]

%\end{theorem}

\begin{proof}[Proof of Theorem 1]
Let us recall a rescaling of the Hermite-Genocchi
%\footnote{You are now picturing a hermit eating a bowl of gnocchi in peace, undisturbed by the existence of weird people doing stuff with B-splines.} 
Formula \cite{zbMATH02166298} for finite differences:
\[ f\left[x_1,\ldots,x_n\right] = \sum_{k=1}^n \frac{f\left(x_k\right)}{\prod_{j\neq k}\left(x_k-x_j\right)}\]
\[=  \frac{1}{(n-1)!}\int_{\Delta_{n-1}} f^{(n-1)}\left(\ilsk{s}{x}\right)\frac{d\mathcal{H}^{n-1}(s)}{\mathcal{H}^{n-1}\left(\Delta_{n-1}\right)}.\]
In particular, for a given $t$, we have 
\[B(t):= \sum_{k=1}^n \frac{\left(x_k-t\right)_+^{n-2} }{\prod_{j\neq k} \left(x_k-x_j\right)} = \left(\cdot  - t\right)_+^{n-2}\left[x_1,\ldots,x_n\right]\]
Since $\left(\frac{d}{dx}\right)^{(n-1)}(x-t)_+^{n-2}= (n-2)!\delta_t$, for any continuous function $g$
\[(n-1)\ilsk{B}{g}=(n-1) \ilsk{t\mapsto (\cdot-t)_+^{n-2}\left[x_1,\ldots,x_n\right] }{g} \] 
\[=(n-1)\int_\R g(t)\frac{1}{(n-1)!}\int_{\Delta_{n-1}} (n-2)!\delta_t \left(\ilsk{s}{x}\right)\frac{d\mathcal{H}^{n-1}(s)}{\mathcal{H}^{n-1}\left(\Delta_{n-1}\right)}d t\] 
\[=\int_{\Delta_{n-1}} g\left(\ilsk{s}{x}\right)\frac{d\mathcal{H}^{n-1}(s)}{\mathcal{H}^{n-1}\left(\Delta_{n-1}\right)}d t\] 
\[=  \E g\left(\ilsk{\mathrm{Unif}\Delta_{n-1}}{x}\right)= \int_\R g(s)\mathrm{PDF}_{\ilsk{\mathrm{Unif}\Delta_{n-1}}{x}}(s)ds,\]
so $(n-1)B$ is precisely the density of $\ilsk{\mathrm{Unif}\Delta_{n-1}}{x}$. But
\[\ilsk{\mathrm{Unif}\Delta_{n-1}}{x}\sim  \frac{\sum x_k P_k}{\sum P_k}= \frac{1}{n}\frac{\sum x_k\left(P_k-1\right)}{1+\frac{1}{\sqrt{n}}\cdot \frac{\sum \left(P_k-1\right)}{\sqrt{n}}}=\frac{1}{n} \frac{Q_1}{1+n^{-1/2}Q_2}.\]
Therefore 
\[B\left(\frac{t}{n}\right)= \frac{1}{n-1}\mathrm{PDF}_{\frac{1}{n} \frac{Q_1}{1+n^{-1/2}Q_2}}\left(\frac{t}{n}\right)= \frac{n}{n-1}\mathrm{PDF}_{\frac{Q_1}{1+n^{-1/2}Q_2}}\left(t\right).
\]
Ultimately, for $|t|<n$,
\[\left|t^p \left(\frac{(-1)^q}{\sqrt{2\pi}}\he{q}{t}e^{-t^2/2} - \d_t^q B\left(\frac{t}{n}\right)
\right)  \right|= \left|t^p \d_t^q\left( B\left(\frac{t}{n}\right)- \mathrm{PDF}_{N(0,1)}(t)\right)\right|\]
\[\leq \frac{n}{n-1}\left|t^p \d_t^q\left( \mathrm{PDF}_{\frac{Q_1}{1+n^{-1/2}Q_2}}\left(t\right)- \mathrm{PDF}_{N(0,1)}(t)\right)\right|+ \frac{1}{n-1}\left|t^p \d_t^q\mathrm{PDF}_{N(0,1)}(t)\right|\]
\[\lesssim_{p,q} \left|t^p \d_t^q\left( \mathrm{PDF}_{\frac{Q_1}{1+n^{-1/2}Q_2}}\left(t\right)- \mathrm{PDF}_{\frac{N_1}{1+n^{-1/2}N_2}}(t)\right)\right|\]
\[+\left|t^p \d_t^q\left( \mathrm{PDF}_{\frac{N_1}{1+n^{-1/2}N_2}}(t)-\mathrm{PDF}_{N(0,1)}(t)\right)\right|+ \frac{1}{n-1}\left|t^p \d_t^q\mathrm{PDF}_{N(0,1)}(t)\right|\]
\[\lesssim  m^3 + n^{-1/2}+ n^{-1}\eqsim m^3,\]
where in the last step we used Theorem \ref{quotqcloseton}, Theorem \ref{puregauss} and a trivial bound for the Gaussian density. This concludes the proof for $|t|\leq n$.\par
If $|t|>n$, then $|t/n|>1$, which places $t/n$ outside of the support of $B$, which is $\left[\min x_k,\max x_k\right]\subset \max\left|x_k\right|\cdot [-1,1]\subset \sum x_k^2\cdot [-1,1]\subset [-1,1]$. If $n$ is big enough (depending on $p,q$), then $t\mapsto \left|t^{p+\frac{1}{2}}\he{q}{t}e^{-t^2/2}\right|$ is decreasing for $t>n$, so 
\[
\left|t^{p}\he{q}{t}e^{-t^2/2}\right|\leq t^{-\frac{1}{2}} n^p\he{q}{n}e^{-n^2/2}\leq n^{-1/2} \sup_{\nu\geq 1}\left|\nu^p\he{q}{\nu}e^{-\nu^2/2} \right|\lesssim_{p,q} m^3.
\]
\end{proof}

\begin{proof}[Proof of Corollary 2.]
If not for the lack of the factor $c_{n,r}:=\frac{r!\binom{n-2}{r}}{n^r}$ in front of the sum, this would be simply a restatement of Theorem 1. Since $r!\binom{n-2}{r}= (n-2)(n-3)\ldots (n-2-(r-1))=: (n-2)^{\underline{r}}$ is a monic polynomial in $n$ of degree $r$, 
\[ \frac{1}{c_{n,r}}-1= \frac{n^r}{(n-2)^{\underline{r}}}-1 = \frac{n^r- (n-2)^{\underline{r}}}{(n-2)^{\underline{r}}} = \frac{\text{poly of deg}\leq r-1}{\text{poly of deg}=r}(n) \lesssim_r n^{-1}.\]
Denoting the Hermite functions by $h_r(t)= (-1)^r e^{t^2/2}\d_t^r e^{-t^2/2}$, 
\[ \sup_t \left| t^p \d_t^q \left((-1)^r h_r(t) - \sum_{k=1}^n \frac{\left(x_k-\frac{t}{n}\right)_{+}^{n-2-r} }{\prod_{j\neq k} \left(x_k-x_j\right)}\right)\right|=\left\|h_r(t)-\frac{1}{c_{n,r}}\d_t^r B(t/n)\right\|_{\mathcal{S}^{p,q,\infty}} \]
\[\leq  \frac{1}{c_{n,r}} \left\|h_r(t)- \d_t^r B(t/n)\right\|_{\mathcal{S}^{p,q,\infty}}+ \left|\frac{1}{c_{n,r}}-1\right| \left\|h_r\right\|_{\mathcal{S}^{p,q,\infty}}.\]
The first summand is bounded by $m^3$, because $\frac{1}{c_{n,r}}=1+O(n^{-1})$, and $$\left\|h_r(t)- \d_t^r B(t/n)\right\|_{\mathcal{S}^{p,q,\infty}} = \left\|h_0(t)- B(t/n)\right\|_{\mathcal{S}^{p,q+r,\infty}}\lesssim_{p,q,r} m^3$$ by definition of $\mathcal{S}$ and by Theorem 1. The second one is simply $\lesssim_r n^{-1} \cdot \|h_0\|_{\mathcal{S}^{p,q+r,\infty}}\lesssim_{p,q,r} n^{-1}\leq m^3$.

\end{proof}

\section{Application for Laguerre polynomials}

Since the Fourier transform\footnote{which will be normalized as $\hat{f}(\xi)=\int_\R f(x)e^{-ix\xi}dx$ here} is continuous on the Schwartz space, we have 
\[\left\|\mathcal{F}_t\left( B(t/n)-\rho(t)\right)\right\|_{\mathcal{S}^{p,q,\infty}}=O_{p,q}(m^3).\]
We are able to explicitly calculate the Fourier transforms at hand. \par
Recall that $\d_t^{n-2}(x_k-t)^{n-2}_+= (-1)^{n-2}(n-2)!\mathmybb{1}_{(-\infty,x_k)}$ and $\d_t^{n-1}(x_k-t)^{n-2}_+= (-1)^{n-1}(n-2)!\delta_{x_k}$. Therefore, denoting $W(x)=\prod_{k=1}^n (x-x_k)$,
\[\d_t^{n-1}B(t)= (-1)^{n-1}(n-2)! \sum_k \frac{\delta_{x_k}}{W'(x_k)}.\]
By the usual $\mathcal{F}_t(\d_tf(t))(\xi)= i\xi \mathcal{F} f(\xi)$, 
\[ \mathcal{F}B(\xi)= \frac{1}{(i\xi)^{n-1}} \mathcal{F}_t(\d_t^{n-1}B(t))(\xi)= \frac{(-1)^{n-1}(n-2)!}{(i\xi)^{n-1}}\sum_k \frac{\mathcal{F}\delta_{x_k}(\xi)}{W'(x_k)} \]
\[= \frac{(-1)^{n-1}(n-2)!}{(i\xi)^{n-2}}\sum_k \frac{e^{-i\xi x_k}}{W'(x_k)}.\]
Thus,
\[
\mathcal{F}_t\left(B(t/n)\right)(\xi)= n\mathcal{F}B(n\xi)= \frac{(-1)^{n-1}(n-2)!\cdot n}{(in\xi)^{n-1}}\sum_k \frac{e^{-in\xi x_k}}{W'(x_k)}\]
\begin{equation}\label{ftr}= \frac{(n-2)!}{n^{n-2}\xi^{n-1}}i^{n-1}\sum_k \frac{e^{-in\xi x_k}}{W'(x_k)}
\end{equation}

Remembering that $\mathcal{F}\rho(\xi)=e^{-\xi^2/2}$ and $\d^r \rho = (-1)^r \he{r}{\cdot} \rho$ and the continuity of $\d$ in the Schwartz space,
\begin{equation}\label{fou}
    \|\frac{(n-2)!i^{n-1}}{n^{n-2}}\d_{\xi}^r\xi^{-(n-1)}\sum_k \frac{e^{-in\xi x_k}}{W'(x_k)} - (-1)^r \he{r}{\xi} \rho(\xi)\|_{\mathcal{S}^{p,q,\infty}}=O_{p,q}(m^3).
    \end{equation}
Now we are going compute the $r$-th derivative in this formula explicitly. As usual, $x^{\underline{y}}= x(x-1)\cdot\ldots\cdot (x-(y-1))$ %(sometimes denoted by $(x)_n$) 
and $x^{\overline{y}}=x(x+1)\cdot\ldots\cdot (x+y-1)$ %(sometimes denoted by $(x)^n$) 
stand for the falling and rising factorial respectively.
%with disdain for any and all usage of $(x)_n$ for the \emph{rising} factorial under the guise of Pochhammer symbol. 
Note that $(-x)^{\underline{y}}= (-1)^y x^{\overline{y}}$. 

\[\d_{\xi}^r\xi^{-(n-1)}\sum_k \frac{e^{-in\xi x_k}}{W'(x_k)}=\sum_{j=0}^r\binom{r}{j} \d_{\xi}^{r-j}\xi^{-(n-1)}\d_{\xi}^j \sum_k \frac{e^{-in\xi x_k}}{W'(x_k)}\]
\[=\sum_{j=0}^r\binom{r}{j} (-(n-1))^{\underline{j}}\xi^{-(n-1)-j}(-inx_k)^{r-j} \sum_k \frac{e^{-in\xi x_k}}{W'(x_k)}\]
\[=\sum_k \frac{e^{-in\xi x_k}}{W'(x_k)} \sum_{j=0}^r\binom{r}{j}(-(n-1))^{\underline{j}} \xi^{-(n-1)-j}(-inx_k)^{r-j}\]
\[=\xi^{-(n-1)}(-in)^r\sum_k \frac{x_k^r e^{-in\xi x_k}}{W'(x_k)} \sum_{j=0}^r\binom{r}{j}(-(n-1))^{\underline{j}} (n\xi x_k)^{-j}\]
\[=\xi^{-(n-1)}(-in)^r\sum_k \frac{x_k^r e^{-in\xi x_k}}{W'(x_k)} \sum_{j=0}^r\frac{1}{j!} r^{\underline{j}}(-(n-1))^{\underline{j}} (n\xi x_k)^{-j}\]
\[= \xi^{-(n-1)}(-in)^r\sum_k \frac{x_k^r e^{-in\xi x_k}}{W'(x_k)} \sum_{j=0}^r\frac{1}{j!} (-r)^{\overline{j}}(n-1)^{\overline{j}} \left(\frac{1}{n\xi x_k}\right)^{j}\]
\[= \xi^{-(n-1)}(-in)^r\sum_k \frac{x_k^r e^{-in\xi x_k}}{W'(x_k)} \,_2 F_0\left(-r,n-1;\frac{1}{n\xi x_k}\right).\]\par

Since $\,_2 F_0\left(-r,n-1;\frac{1}{n\xi x_k}\right)=
r!(n\xi x_k)^{-r}L^{(-n-r+1)}_r(-n\xi x_k)$, (cf.\cite{DGV}) plugging the above into \eqref{fou} we get Corollary 3.

To obtain yet another corollary, this time concerning the probability distribution used in the proof of our main result, let us continue transforming formula 
\eqref{ftr}. We have further

\[\mathcal{F}_t\left(B(t/n)\right)(\xi)=\frac{(n-2)!}{n^{n-2}\xi^{n-1}}i^{n-1} \left(y\mapsto e^{-in\xi y}\right)\left[x_1,\ldots,x_n\right]\]
\[\overset{\text{Hermite-Genocchi}}{=} \frac{(n-2)!}{n^{n-2}\xi^{n-1}}i^{n-1}\frac{1}{(n-1)!}\E \left(\d^{n-1}_y e^{-in\xi y}\right)(\ilsk{x}{\mathrm{Unif}\Delta_{n-1}})\]
\[=\frac{i^{n-1}}{(n-1)\cdot n^{n-2}\xi^{n-1}}(-in\xi)^{n-1}\E e^{-in\xi\ilsk{x}{\mathrm{Unif}\Delta_{n-1}}}\]
\[=\frac{n}{n-1} \E e^{-in\xi\ilsk{x}{\mathrm{Unif}\Delta_{n-1}}},\]
which is not surprising, because we proved that
$(n-1)\frac{1}{n}B(t/n)$ is the density of the distribution $n\ilsk{x}{\mathrm{Unif}\Delta_{n-1}}$, so the Fourier transform of $B(t/n)$ is nothing but its characteristic function (up to $\xi\mapsto -\xi$). Since 
$\mathcal{F}\rho(\xi)=e^{-\xi^2/2}$,
 by taking the real and imaginary part and keeping in mind that for big enough $n$ the factor $\frac{n}{n-1}$ is close to 1, we infer the following. 

\begin{corollary}
Let $\mathrm{Unif}\Delta_{n-1}$ be a random vector uniformly distributed on the standard $(n-1)$-dimensional simplex and $p,q$ and $x$ satisfy the assumptions of Theorem 1. Then
 \[\left\|\E\cos\left(n\xi\ilsk{x}{\mathrm{Unif}\Delta_{n-1}}\right)-e^{-\xi^2/2}\right\|_{\mathcal{S}^{p,q,\infty}}=O_{p,q}(m^3),\]
 \[\left\|\E\sin\left(n\xi\ilsk{x}{\mathrm{Unif}\Delta_{n-1}}\right)\right\|_{\mathcal{S}^{p,q,\infty}}=O_{p,q}(m^3).\]
\end{corollary}

%\begin{thebibliography}{9}
\bibliographystyle{plain}
\bibliography{lit}

%\end{thebibliography}
	
\end{document}